\documentclass[a4paper]{amsart}
\usepackage{amsmath,amssymb,amsthm,enumerate}
\title[Renormalized Chern-Gauss-Bonnet formula]
{{\bf Renormalized Chern-Gauss-Bonnet formula for 
complete K\"{a}hler-Einstein metrics}}
      
\author{TAIJI MARUGAME}
\date{}
\address{Graduate School of Mathematical Sciences, The University of Tokyo, 3-8-1 Komaba, Meguro, Tokyo 153-8914, Japan}
\email{marugame@ms.u-tokyo.ac.jp}
\newcommand\C{\mathbb{C}}
\newcommand\R{\mathbb{R}}
\newcommand\Sp{\mathbb{S}}
\newcommand\J{\mathcal{J}}
\newcommand\scal{{\rm Scal}}
\renewcommand\a{\alpha}
\renewcommand\b{\beta}
\newcommand\g{\gamma}
\renewcommand\d{\delta}
\newcommand\e{\varepsilon}
\renewcommand\r{\rho}
\newcommand\U{\Upsilon}
\renewcommand\th{\theta}
\newcommand\Th{\Theta}
\newcommand\ol{\overline}
\newcommand\z{\zeta}
\newcommand\pa{\partial}

\newtheorem{mainthm}{Theorem}
\newtheorem{thm}{Theorem}[section]
\newtheorem{dfn}[thm]{Definition}
\newtheorem{prop}[thm]{Proposition}

\newtheorem{lem}[thm]{Lemma}
\newtheorem{cor}[thm]{Corollary}

\makeatletter
\@addtoreset{equation}{section}

\makeatother
\begin{document}
\maketitle
\begin{abstract}
We present a renormalized Gauss-Bonnet formula for a strictly pseudoconvex manifold with a complete K\"{a}hler metric given by a globally defined potential function near the boundary.  When the metric is asymptotically Einstein, the boundary contribution in the formula is explicitly written down in terms of the pseudo-hermitian geometry of the boundary and is shown to be a global CR invariant.  The CR invariant generalizes the Burns-Epstein invariant on 3-dimensional CR manifolds.
\end{abstract}

\section{Introduction}

In recent years there has been much work on the renormalization of curvature integrals for complete metrics. In conformal geometry, the renormalized volume plays an important role in the study of a conformally compact Einstein metric and its conformal infinity, and is motivated by the AdS/CFT correspondence in theoretical physics; see \cite{G2}. Several authors derive Gauss-Bonnet formulas incorporating the renormalized volume; see \cite{An}, 
\cite{CQY}, \cite{He}, \cite{S}. These works are based on the cut-off method with respect to the special defining function, and deep analysis of asymptotic behavior of the metric near the boundary. 
 
 In complex geometry, the corresponding setting is the complete K\"{a}hler-Einstein manifold $(X,g)$ and the CR structure on the boundary $M=\pa X$.  The construction of such metrics goes back to 1980 (\cite{Fe2},\cite{CY}) and the renormalized Gauss-Bonnet formula for strictly pseudoconvex domains $\C^{n+1}$ was derived by Burns and Epstein \cite{BE2} in 1990. 
In contrast to the conformal case, their method is much more algebraic: the Bochner tensor $\Th$, or the trace-free part of the curvature of $g$, turns out to be continuous up to the boundary and one can define renormalized Chern classes $c_k(\Th)$ as invariant polynomials of $\Th$. The top Chern form $c_{n+1}(\Th)$ can be integrated over $X$ and defines a characteristic number of $X$.  The renormalized Gauss-Bonnet formula takes the form
\begin{equation}\label{RGBF}
\int_X c_{n+1}(\Th)=\chi(X)+\mu(M),
\end{equation}
where $\mu(M)$ is a correction term determined by the CR geometry of the boundary.
In the $n=1$ case, they derived $\mu(M)$ with a rather direct computation:
\begin{equation}\label{mudim3}
\mu(M)=\frac{1}{4\pi^2}\int_{\pa X} \Bigl(|A|^2-\frac{1}{4}\scal^2\Bigr)\th\wedge d\th,
\end{equation}
where $|A|^2$ is the squared norm of the Tanaka-Webster (TW) torsion and $\scal$ is the TW scalar curvature with respect to $\th$, which is assumed to be pseudo-Einstein (see \S\ref{section-pe}).
The integral on the right-hand side can be generalized to abstract 3-dimensional CR manifold $M$ with trivial CR holomorphic tangent bundle (\cite{BE1}) and $\mu(M)$ is now called the Burns-Epstein invariant; see \cite{CL} and \cite{BH} for  further generalization of the invariant in real 3-dimensions.

In the case $n\ge2$, the renormalized Gauss-Bonnet formula is derived from  a Chern-Simons type transgression formula defined on the Cartan bundle over a CR manifold. The boundary term $\mu(M)$ is defined as a paring of the transgression form with a cycle, called the homological section, of the bundle.  Due to this topological procedure, it is not easy to relate $\mu(M)$ with the local CR or pseudo-hermitian geometry of $M$.  Also, the construction of the homological section is done by using the global coordinates of $\C^{n+1}$ and their result is confined to this setting.

We here derive a renormalized Gauss-Bonnet formula for a strictly pseudoconvex manifold $X$ with a complete Einstein metric of the form $g=\pa\ol{\pa}\log(-1/\r)$ near the boundary $M$, where $\r$ is a defining function of $M$. 
In the formula, $\mu(M)$ is given explicitly in terms of the TW curvature and torsion of the contact form $\th=(\sqrt{-1}/2)(\ol{\pa}\r-\pa\r)|_{TM}$.
In particular, in case $\dim_{\mathbb{R}}M=5$, we have
\begin{equation}\label{mudim5}
\mu(M) 
=\frac{1}{16\pi^3}\int_{M}\Bigl(\frac{1}{54}(\scal)^3-\frac{1}{12}|R|^2\scal+R_{\a\ol{\b}\g\ol{\d}}A^{\a\g}A^{\ol{\b}\ol{\d}}\Bigr)\th\wedge (d\th)^2,
\end{equation}
where $R$ and $A$ are respectively the TW curvature and torsion.
For higher dimensions, the integrand is given by a linear combination of complete contractions of tensor products of several 
$R$ and $A$.

We significantly simplify and generalize the derivation of $\mu(M)$ in \cite{BE2} by introducing a local frame of the holomorphic tangent bundle $T^{1,0}X$ which is compatible with the CR structures of the foliation defined by the level sets of $\rho$.  We can use the local frame to derive a transgression form for $c_{n+1}(\Th)$, which is independent of the choice of the frame and thus globally defined on $X$.
It then becomes possible to apply the classical method in Chern's proof of the Gauss-Bonnet theorem \cite{Ch} to prove the renormalized version \eqref{RGBF}.
This argument is irrelevant to the Einstein condition and our first main theorem can be formulated as follows:

\begin{mainthm}\label{first-mainthm}
Let $X$ be an $(n+1)$-dimensional compact complex manifold with strictly pseudoconvex boundary $M$. Assume that $X$ admits a complete hermitian metric $g$ which is of the form $g=\pa\ol{\pa}\log(-1/\r)$ near the boundary for a defining function $\r$ of $M$. Let $\Th$ be the Bochner tensor of $g$ and set $\th=(\sqrt{-1}/2)(\ol{\pa}\r-\pa\r)|_{TM}$. Then the renormalized Chern form for $g$ satisfies
\begin{equation}\label{rgbintro}
\int_X c_{n+1}(\Th)=\chi(X)+\mu(M,\rho).
\end{equation}
The boundary term is given by the integral
$$
\mu(M,\rho)=\int_M F(R,A,r)\,\th\wedge(d\th)^n,
$$
where $F(R,A,r)$ is an invariant polynomial in the components of the TW curvature $R$, TW torsion $A$ and the transverse curvature $r$ of $\rho$.\end{mainthm}

While the explicit form of $F(R,A,r)$ is not easy to write down, in the course of the proof, we obtain an expression 
$$
F(R,A,r)=\Pi(\th_i{}^j,\Theta),
$$
where $\Pi$ is a transgression form given explicitly as a polynomial in $\Th$ and the renormalized connection form $\th_i{}^j$ of $g$; see Theorem \ref{transgression}.  

 We now specialize Theorem \ref{first-mainthm} to the case for asymptotically Einstein metric, which is equivalent to imposing the approximate complex Monge-Amp\`{e}re equation for the defining function $\rho$.

\begin{mainthm}\label{second-mainthm} Let $X$ and $g$ be as in Theorem \ref{first-mainthm}.
Assume that $\r$ satisfies the approximate complex Monge-Amp\`{e}re equation.  Then $\mu(M,\r)$ is independent of the choice of such $\r$ and defines a CR invariant $\mu(M)$ of $M$. Moreover, it has the form
$$
\mu(M)=\int_M F(R,A)\,\th\wedge(d\th)^n,
$$
where $F(R,A)$ is an invariant polynomial in the components of the TW curvature $R$ and  torsion $A$.
In the case $n=1$ and $2$, $F(R,A)$ is respectively given by
\eqref{mudim3} and
\eqref{mudim5}.
\end{mainthm}

Note that the existence of a global defining function $\r$
satisfying the asymptotic Monge-Amp\`{e}re equation is equivalent to the existence of a pseudo-Einstein contact form.  Thus we may say that the CR invariant $\mu(M)$ is defined on pseudo-Einstein manifolds; see Theorem \ref{HPTthm}.  It is also known that the boundary of a strictly pseudoconvex domain in a Stein manifold admits a pseudo-Einstein contact form \cite{CC}; so Theorem 2 can be applied to such a domain.

While the formula of $\mu(M)$ for low dimensions is explicitly given, it is still not easy to evaluate the integral.  We give here two simple examples for which we can write down $\mu(M)$ with a help of symmetry. First, let $(L,h)\to Y$ be a hermitian line bundle over a compact 2-dimensional complex manifold. 
We assume that the unit disc bundle $X\subset L$ is strictly pseudoconvex, or equivalently, the curvature $\operatorname{curv}(h)$ of $h$ is negative definite.  If we  further assume that $-\operatorname{curv}(h)$ defines an K\"ahler-Einstein metric $\underline{g}$, then the contact from defined from $\rho=\log h(v,v)$ is pseudo-Einstein; so that the assumption of Theorem \ref{second-mainthm} is satisfied.  In this setting, we can write $\mu(M)$ for $M=\pa X$ as follows:
\begin{equation}
\label{muM-example}
 \mu(M)=\frac{\sigma}{36}\Bigl(\chi(Y)-\frac{1}{8\pi^2}\int_Y |{\rm Weyl}|^2vol_{\underline{g}}\Bigr), 
\end{equation}
 where $\sigma$ and {\rm Weyl} are respectively the scalar and the Weyl curvatures of $\underline{g}$ as a Riemannian metric.
The second example is the boundary of a Reinhardt domain $\Omega_r=\{(w^0, w^1, w^2)\in\C^3\ |\ \sum(\log|w^i|)^2<r^2\}$,   on which $O(3)\times\mathbb{T}^3$ acts transitively as local CR diffeomorphisms. We can take a pseudo-Einstein contact form 
which is invariant under this action, and the computation of $\mu(\pa\Omega_r)$ is reduced to evaluating $\Pi$ at a single point of $\pa\Omega_r$. The result of the computation is 
\begin{equation}\label{mu-Reinhardt}
\mu(\pa\Omega_r)=-\frac{20\pi}{27}\frac{1}{r^3}. 
\end{equation}
By Fefferman's theorem, if two strictly pseudoconvex domains are biholomorphic, then their boundaries are CR diffeomorphic. Since 
$\mu(\pa\Omega_r)$ is a CR invariant, it follows from (\ref{mu-Reinhardt}) that $\Omega_r$ and $\Omega_{r^\prime}$ are biholomorphic if and only if $r=r^\prime$.

Finally, let us mention another global CR invariant defined
from the $Q$-prime curvature on pseudo-Einstein manifolds (\cite{CaY}, \cite{Hi2}). $Q$-prime curvature is a local invariant of a pseudo-Einstein contact form; while it is not a local CR invariant, its integral is shown to be a global CR invariant, which is called the total $Q$-prime curvature
and denoted by $\overline{Q'}(M)$. For 3-dimensional manifolds, the total $Q$-prime curvature agrees with the Burns-Epstein invariant. However, for higher dimensions, 
$\mu(M)$ and  $\overline{Q'}(M)$ have different properties. For example, if $M$ is the boundary of a complete K\"ahler-Einstein manifold $(X,g)$, then
$\overline{Q'}(M)$ can be formulated as a renormalized volume of $(X,g)$.  The details will  appear in our joint paper with K. Hirachi and Y. Matsumoto \cite{HMM}.

This paper is organized as follows. In \S2, we review basic notions of the geometry of CR manifolds and complex manifolds with boundary. In \S3, we define the renormalized connection by following Burns and Epstein. Then we prove Theorem \ref{first-mainthm} by constructing a transgression of the renormalized Chern form.  In \S4, we introduce the approximate complex Monge-Amp\`{e}re equation for a defining function, and prove Theorem \ref{second-mainthm}. In the last section, \S5, is devoted to the proofs of the examples \eqref{muM-example} and \eqref{mu-Reinhardt}.

\medskip
\medskip

{\it Notations.} We use Einstein's summation convention and raise or lower indices by the Levi form. 

\begin{itemize}
\item
The lowercase Latin indices $i,j,k,l$  run from 0 to $n$.
\item
The lowercase Greek indices $\alpha,\beta,\gamma$ run from 1 to $n$. 
\end{itemize}

\medskip

{\bf Acknowledgement.} The author is grateful to his advisor Professor Kengo Hirachi for introducing him to this work and for various suggestions. He thanks Doctor Yoshihiko Matsumoto for helpful comments in the seminars. He also would like to express his gratitude to Professor Daniel Burns and Professor Paul Yang for having interest in this work and suggesting some related problems. 

\section{Pseudo-hermitian geometry}
\subsection{Pseudo-hermitian structures} 
Let $M$ be a $(2n+1)$-dimensional $C^\infty$ manifold. A {\it CR structure} on $M$ is a pair $(H,J)$, where
$H \subset TM$ is a subbundle of rank $2n$ called the {\it Levi distribution} and $J$ is an almost complex structure on $H$. The complexification of $H$ has the eigenspace decomposition with respect to $J$:
\begin{equation*}
\C \otimes H =T^{1,0}M\oplus T^{0,1}M,\ T^{0,1}M=\ol{T^{1,0}M}.
\end{equation*}
Here, $T^{1,0}M$ is the eigenspace corresponding to the eigenvalue $\sqrt{-1}$, and called ({\it CR}) {\it holomorphic tangent bundle} of $M$. We assume that $T^{1,0}M$ satisfies the following integrability condition: 
\begin{equation*}
[\Gamma (T^{1,0}M),\ \Gamma (T^{1,0}M)]\subset \Gamma (T^{1,0}M).
\end{equation*}

Suppose that $M$ is orientable. Then there exists a global nonvanishing real 1-form $\th$ which annihilates $H$.
 For a choice of such $\th$,\ the {\it Levi form} is defined by 
\begin{equation*} 
L_{\th}(V,\ol{W})=-\sqrt{-1}d\th (V\wedge \ol{W})\ \ \ \ \ V,W \in T^{1,0}M.
\end{equation*}
A CR structure is said to be {\it strictly pseudoconvex }if $\th$ can be taken so that $L_{\th}$ is positive definite. 
Such a $\th$ is called a {\it pseudo-hermitian structure} or a {\it contact form}. For another contact form $\widehat{\th}=e^{\U}\th\ \ (\U \in C^{\infty}(M))$, the Levi form transforms as $L_{\widehat{\th}}=e^{\U}L_{\th}$, so the conformal class of the Levi form is well-defined for a CR structure. In this point of view, one can regard CR structures as {\it ``pseudo-conformal structures"}. 
Fixing a contact form, we can define a canonical linear connection on $M$ called the {\it Tanaka-Webster $($TW$)$ connection} \cite{W}, \cite{T}, which is described as follows: Let $T$ be the {\it characteristic vector field}, which is the real vector field characterized by 
\begin{equation*}
\th (T)=1,\ T \lrcorner d\th =0.
\end{equation*}
A local frame $\{T, Z_{\a}, Z_{\ol{\a}}\}$ of $\C\otimes TM$ is called an {\it admissible frame} when $\{Z_{\a}\}$ is a local frame of 
$T^{1,0}M$ and $Z_{\ol{\a}}=\ol{Z_{\a}}$. The dual frame $\{\th, \th^\a, \th^{\ol{\a}}\}$ of an admissible frame is called an {\it admissible coframe}. In an admissible coframe, $d\th$ can be written as 
\begin{equation*}
d\th =\sqrt{-1} h_{\a\ol{\b}}\th^{\a}\wedge \th^{\ol{\b}}
\end{equation*}
and the TW connection $\nabla$ is defined by 
\begin{equation*}
\nabla T =0,\ \ \nabla Z_{\a}={\omega _{\a}}^{\b}\otimes Z_{\b},\ \ \nabla Z_{\ol{\a}}={\omega _{\ol{\a}}}^{\ol{\b}}\otimes Z_{\ol{\b}} \ \ ({\omega _{\ol{\a}}}^{\ol{\b}}=\ol{{\omega _{\a}}^{\b}})
\end{equation*}
with structure equations
\begin{gather*}
d\th^{\a}=\th^{\b}\wedge {\omega _{\b}}^{\a}+ {A^{\a}}_{\ol{\b}} \th \wedge \th^{\ol{\b}} \\
dh_{\a\ol{\b}}={\omega _{\a}}^{\g}h_{\g\ol{\b}}+h_{\a\ol{\g}}{\omega _{\ol{\b}}}^{\ol{\g}}.
\end{gather*}
The tensor ${A^{\a}}_{\ol{\b}}$ is called the {\it Tanaka-Webster $($TW$)$ torsion}, and satisfies $A_{\a\b}=A_{\b\a}$. We will use the symbol $T$ as the index of tensors corresponding to the direction $T$, and will denote components of covariant derivatives with indices preceded by a comma, e.g. $A_{\a\b,\g}=\nabla_{\g}A_{\a\b}$. We usually omit the comma for functions. For a function $f$, one can write as $df=f_{T}\th +\partial_{b}f+{\ol{\partial}}_{b}f$, where $\partial_{b}f:=f_{\a}\th^{\a}$ and $ {\ol{\partial}}_{b}f:=f_{\ol{\a}}\th^{\ol{\a}}$. The {\it sub-Laplacian} $\Delta_b$ is defined by $\Delta_b f=-{f_\g}^\g-{f_{\ol{\g}}}^{\ol{\g}}$.  A function $f$ is called a {\it CR holomorphic function} if ${\ol{\partial}}_{b}f=0$, and a real valued function is CR pluriharmonic if it is locally the real part of a CR holomorphic function.  

The curvature form of the TW connection ${\Omega _{\a}}^{\b}=d{\omega _{\a}}^{\b}-{\omega _{\a}}^{\g}\wedge {\omega _{\g}}^{\b}$ satisfies the structure equation
\begin{multline*}
{\Omega _{\a}}^{\b}={{R _{\a}}^{\b}}_{\g\ol{\d}}\ \th^{\g}\wedge \th^{\ol{\d}}+{A_{\a\g,}}^{\b}\th^{\g}\wedge \th
-{A^{\b}}_{\ol{\g},\a}\th^{\ol{\g}}\wedge \th \\
-\sqrt{-1}\ A_{\a\g}\th^{\g}\wedge \th^{\b}
+\sqrt{-1}\ h_{\a\ol{\g}}{A^{\b}}_{\ol{\d}}\th^{\ol{\g}}\wedge \th^{\ol{\d}}. \label{curvature}
\end{multline*}
We will use the Levi form $h_{\a\ol{\b}}$ to raise and lower the indices. The tensor $R _{\a\ol{\b}\g\ol{\d}}$, called the {\it Tanaka-Webster $($TW$)$ curvature}, has the following symmetry:
\begin{equation*}
R _{\a\ol{\b}\g\ol{\d}}=R _{\g\ol{\b}\a\ol{\d}},\ \ R _{\a\ol{\b}\g\ol{\d}}=R _{\a\ol{\d}\g\ol{\b}},\ \ 
R _{\a\ol{\b}\g\ol{\d}}=R _{\ol{\b}\a\ol{\d}\g}(:=\ol{R _{\b\ol{\a}\d\ol{\g}}}).
\end{equation*}
Contracting the curvature tensor, we define the Ricci and scalar curvatures:
\begin{equation*}
{\rm Ric}_{\g\ol{\d}}={{R _{\a}}^{\a}}_{\g\ol{\d}},\ \ \scal ={{\rm Ric}_{\g}}^{\g}.
\end{equation*}
\subsection{Strictly pseudoconvex manifolds} Let $X$ be a relatively compact domain with smooth boundary in a complex manifold $\widetilde{X}$. A defining function of $X$ is a smooth real function on $\widetilde{X}$ such that  
$X=\{\r<0\}$ and $d\r\neq 0$ at $\pa X$. The boundary $M=\pa X$ has a natural CR structure induced by the complex structure $J$ of $\widetilde{X}$; the Levi distribution $H$ is the maximal $J$-invariant subbundle of $TM$, and the restriction of $J$ defines an almost complex structure on $H$. The integrability condition for the CR structure follows from that of $J$. We say $X$ is a strictly pseudoconvex manifold if the Levi form of the contact form$(\sqrt{-1}/2)(\ol{\pa}\r-\pa\r)|_{TM}$ is positive definite.

Let us recall the Graham-Lee connection \cite{GL} for a defining function $\r$. When the boundary $M=\pa X$ is strictly pseudoconvex, there exists a unique (1,0)-vector field $\xi$ near $M$ satisfying 
\begin{equation}\label{xi}
\xi\r=1,\ \ \xi\lrcorner\,\pa\ol{\pa}\r=r\ol{\pa}\r
\end{equation} 
with a real function $r$. We call $r$ the {\it transverse curvature} of $\r$. Let $\{\th^{0},\th^{\a}\}$ be the dual frame of a local (1,0)-frame $\{W_{0}=\xi\ ,W_{\a}\}$ with $\{W_{\a}\}$ 
a frame for $\mathcal{H}:=$Ker$\partial\r$. Then $\th^{0}=\partial \r$ and one can write as
\begin{equation} \label{ambient}
\partial\ol{\partial}\r =h_{\a\ol{\b}} \th^{\a}\wedge\th^{\ol{\b}}+r\partial\r\wedge\ol{\partial}\r.
\end{equation}
For sufficiently small $\e>0$, $M^{-\e}:=\{\r=-\e \}$ is a strictly pseudoconvex CR manifold and the restriction of $\th:=
(\sqrt{-1}/2)(\ol{\pa}\r-\pa\r)$ gives a contact form. From (\ref{ambient}), the Levi form for this contact form is 
$h_{\a\ol{\b}}$ and the restriction of $\{\th, \th^{\a}, \th^{\ol{\a}}\}$ becomes an admissible coframe. 
We write 
\begin{equation*}
\xi =N-\frac{\sqrt{-1}}{2}T
\end{equation*}
with real vector fields $N$ and $T$. Then one can check that
\begin{equation*}
N\r=1,\ \ \th(N)=0,\ \ T\r=0,\ \ \th(T)=1,\ \ T\lrcorner d\th  |_{TM^{-\e}}=0
\end{equation*}
so $T$ is the characteristic vector field on each $M^{-\e}$.

Now the Graham-Lee connection is defined by the following proposition:
\begin{prop}[\cite{GL}]\label{GLcurv}
For a defining function $\r$, there exists a unique linear connection $\nabla$ on $X$ near $M$ with the properties:
\begin{itemize}
  \item[(a)] For any vector fields $Y$ and $Z$ that are tangent to some $M^{-\e}$,
   $\nabla_{Y}Z={\nabla^{-\e}}_{Y}Z$             where $\nabla^{-\e}$ is the TW connection on $M^{-\e}$.
  \item[(b)] $\nabla T=\nabla N=0\ and\ \nabla h_{\a\ol{\b}}=0$.
  \item[(c)] Let $\{W_{\a}\}$ be any frame for $\mathcal{H}$ and $\{\partial\r,\th^{\a}\}$ the $(1,0)$-coframe dual to 
             $\{\xi ,W_{\a}\}$. The connection one-forms ${\varphi _{\a}}^{\b}$, defined by $\nabla                     W_{\a}={\varphi_{\a}}^{\b}\otimes W_{\b}$, satisfy the following structure equation:
             \begin{equation}\label{GLstr}
             d\th^{\a}=\th^{\b}\wedge{\varphi_{\b}}^{\a}-\sqrt{-1}{A^{\a}}_{\ol{\g}}\partial\r\wedge\th^{\ol{\g}}-r^{\a}\partial\r\wedge\ol{\partial}\r+\frac{1}{2}rd\r\wedge\th^{\a}.
             \end{equation}             
\end{itemize}
The curvature form ${\Omega_{\a}}^{\b}=d{\varphi_{\a}}^{\b}-{\varphi_{\a}}^{\g}\wedge{\varphi_{\g}}^{\b}$ is given by
\begin{align*}
{\Omega_{\a}}^{\b}&={{R _{\a}}^{\b}}_{\g\ol{\d}}\ \th^{\g}\wedge \th^{\ol{\d}}+\sqrt{-1}{A_{\a\g,}}^{\b}\th^{\g}\wedge\ol{\partial}\r
+\sqrt{-1}{A^{\b}}_{\ol{\g},\a}\th^{\ol{\g}}\wedge\partial\r \\
&\quad -\sqrt{-1}\ A_{\a\g}\th^{\g}\wedge \th^{\b}
+\sqrt{-1}\ h_{\a\ol{\g}}{A^{\b}}_{\ol{\d}}\th^{\ol{\g}}\wedge \th^{\ol{\d}} \\
&\quad +d\r\wedge \Bigl( r_{\a}\th^{\b}-r^{\b}\th_{\a}+\frac{1}{2}{\d_{\a}}^{\b}r_{\g}\th^{\g}-\frac{1}{2}{\d_{\a}}^{\b}r_{\ol{\g}}\th^{\ol{\g}}\Bigr) \\
&\quad -\frac{1}{2}\Bigl( {r_{\a}}^{\b}+{r^{\b}}_{\a}+2A_{\a\g}A^{\g\b} \Bigr)\partial\r\wedge\ol{\partial}\r,\qquad
\end{align*}
where ${{R _{\a}}^{\b}}_{\g\ol{\d}}$ are the components of the TW curvature tensor, and ${\d_{\a}}^\b$ denotes 
the Kronecker delta.
\end{prop}
Note that the condition (a) follows from the others. One can see that $A_{\a\b}$ above coincides with the TW torsion on each $M^{-\e}$. We will use the index 0 for the covariant differentiation in the direction $\xi$, e.g., $A_{\a\b,0}=A_{\a\b,N}-(\sqrt{-1}/2)A_{\a\b,T}$.  

\section{Proof of Theorem \ref{first-mainthm}}
\subsection{Transgression and Gauss-Bonnet formula}Let $X$ be a strictly pseudoconvex manifold and $\r$ a defining function of $X$. By strict pseudoconvexity, a $(1,1)$-form $\pa\ol{\pa}\log(-1/\r)$ defines a K\"{a}hler metric on $X$ near the boundary $M=\pa X$. We extend this metric to a hermitian metric $g$ on $X$, and denote the connection and the curvature forms of $g$ by ${\psi_i}^{j}$ and ${\Psi_i}^{j}$ respectively. Burns and Epstein \cite{BE2} renormalized these forms by subtracting singularities. The renormalized connection form is defined by 
\begin{equation}\label{defth}
{\th_{i}}^{j}:={\psi_{i}}^{j}+{Y_{i}}^{j}, 
\end{equation}
where 
\begin{equation*}
{Y_{i}}^{j}:=\frac{1}{\r}\bigr( {\d_{i}}^{j}\r_{k}+{\d_{k}}^{j}\r_{i}\bigr)\th^{k}.
\end{equation*}
The renormalized curvature is defined by 
\begin{equation}\label{defW}
{W_{i}}^{j}:={\Psi_{i}}^{j}+{K_{i}}^{j},
\end{equation}
where
\begin{equation*}
{K_{i}}^{j}:=\big( {\d_{i}}^{j}g_{k\ol{l}}+{\d_{k}}^{j}g_{i\ol{l}}\bigr) \th^{k}\wedge\th^{\ol{l}}.
\end{equation*}
The curvature ${\Th_i}^j=d{\th_{i}}^{j}-{\th_{i}}^{k}\wedge{\th_{k}}^{j}$ is called the {\it Bochner tensor} of $g$. Burns and Epstein showed that ${\th_{i}}^{j}$ is continuous up to the boundary and defined the {\it characteristic number } by 
\begin{equation}\label{char-number}
 \int_X c_{n+1}(\Th)=\int_X \det\Bigl(\frac{\sqrt{-1}}{2\pi}{\Th_i}^j\Bigr). 
\end{equation}
We decompose (\ref{char-number}) into the Euler characteristic and a boundary integral by constructing a transgression form for $c_{n+1}(\Th)$. Let $\xi$ be the $(1,0)$-vector field near $M$ characterized by (\ref{xi}). We extend $\xi$ to $X$ so that it has finitely many non-degenerate zero points $\{p_{1},\dots,p_{m}\}$. Take a non-vanishing $(1,0)$-vector filed $V$ on $X\setminus\{p_{1},\dots,p_{m}\}$ which satisfies  $V=\xi$ near $M$ and $V=\xi/|\xi|_g$ near each $p_j$. We extend $\mathcal{H}={\rm Ker}\,\pa\r$ to $X\setminus\{p_{1},\dots,p_{m}\}$ so that $T^{1,0}(X\setminus\{p_{1},\dots,p_{m}\})=\C V\oplus \mathcal{H}$
 holds and we take a local $(1,0)$-frame as $\{W_{0}=V, W_{\a}\}$ according to the decomposition. We construct a differential form $\Pi$ on $\ol{X}\setminus\{p_{1},\dots,p_{m}\}$ such that $d\Pi=c_{n+1}(\Th)$. Then we have 
\begin{equation*}
\int_{X}c_{n+1}(\Th) =-\lim_{\e \to 0}\sum_{j=1}^{m}\int_{\pa B_{\e}(p_{j})}\Pi +\int_{M}\Pi, 
\end{equation*}
where $B_{\e}(p_{j})$ is the ball of radius $\e$ centered at $p_{j}$ with respect to $g$. The first term in the right-hand side is shown to be the sum of the indices of ${\rm Re}\, \xi$, which equals $\chi(X)$ by the 
Poincar\'{e}-Hopf theorem. The second term can be expressed in terms of the Graham-Lee connection. 

The transgression form $\Pi$ is given by the following proposition.
\begin{prop}\label{transgression}
There exists a differential form $\Pi$ on $\ol{X}\setminus\{p_{1},\dots,p_{m}\}$ that satisfies 
\begin{equation*}
c_{n+1}(\Th)=d\Pi.
\end{equation*}
Moreover, $\Pi$ is given by 
\begin{equation*}
\Pi=\frac{1}{n!}\Bigl(\frac{\sqrt{-1}}{2\pi}\Bigr)^{n+1}\sum_{k=0}^{n}\binom{n}{k}(\Phi_{k}^{(0)}-\Phi_{k}^{(1)}),
\end{equation*}
where
\begin{align*}
\Phi_{0}^{(0)}&= \sum_{\sigma ,\tau\in S_{n}} {\rm sgn}(\sigma\tau)\ {\th_{0}}^{0}{\th_{\sigma (1)}}^{0}{\th_{0}}^{\tau (1)}\cdots{\th_{\sigma (n)}}^{0}{\th_{0}}^{\tau (n)}, \\
\Phi_{k}^{(0)}&= \sum_{\sigma ,\tau\in S_{n}} {\rm sgn}(\sigma\tau)\ {\th_{0}}^{0}{\Th_{\sigma (1)}}^{\tau (1)}\cdots{\Th_{\sigma (k)}}^{\tau (k)}\\ \notag
&\qquad\qquad\cdot{\th_{\sigma (k+1)}}^{0}{\th_{0}}^{\tau (k+1)}\cdots{\th_{\sigma (n)}}^{0}{\th_{0}}^{\tau (n)}\quad(1\le k\le n), \\
\Phi_{0}^{(1)}&= \sum_{\sigma ,\tau\in S_{n}} {\rm sgn}(\sigma\tau)\ {\Th_{\sigma(1)}}^{0}{\th_{0}}^{\tau(1)}{\th_{\sigma (2)}}^{0}{\th_{0}}^{\tau (2)}\cdots{\th_{\sigma (n)}}^{0}{\th_{0}}^{\tau (n)}, \\
\Phi_{k}^{(1)}&= \sum_{\sigma ,\tau\in S_{n}} {\rm sgn}(\sigma\tau)\ {\Th_{\sigma(1)}}^{0}{\th_{0}}^{\tau(1)}{\Th_{\sigma (2)}}^{\tau (2)}\cdots{\Th_{\sigma (k+1)}}^{\tau (k+1)} \\ \notag
&\qquad\qquad\cdot{\th_{\sigma (k+2)}}^{0}{\th_{0}}^{\tau (k+2)}\cdots{\th_{\sigma (n)}}^{0}{\th_{0}}^{\tau (n)}\quad(1\le k\le n-1),\\
\Phi_{n}^{(1)}&=0.
\end{align*}
\end{prop}
In the above proposition, $S_k$ denotes the symmetric group of order $k$, and $S_{n+1}$ acts on $\{0,1,\dots ,n\}$ while $S_n$ acts on $\{1,2,\dots,n\}$. We have omitted the wedge product symbols in the formulas.

\begin{proof}
Let us define a linear connection $\widetilde{\nabla}$ on $T^{1,0}(\ol{X}\setminus\{p_{1},\dots,p_{m}\})$ by
\begin{align*}
\widetilde{\nabla}W_0 &=0, \\
\widetilde{\nabla}W_{\a}&={\th_{\a}}^{i}\otimes W_i,
\end{align*}
and set 
\begin{equation*}
{\widetilde{\nabla}}^{(t)}:=t\nabla +(1-t)\widetilde{\nabla},
\end{equation*}
where $\nabla$ is the connection defined by ${\th_{i}}^{j}$. Then the connection form of ${\widetilde{\nabla}}^{(t)}$ is 
\begin{equation*}
\begin{pmatrix}
  t{\th_{0}}^{0} & t{\th_{0}}^{\a} \\
  {\th_{\b}}^{0} & {\th_{\b}}^{\a} 
\end{pmatrix}
\end{equation*}
and the curvature $\Th_{t}$ is given by 
\begin{equation}
 \begin{array}{rcl}\label{curv}
{(\Th_{t})_{0}}^{0}&=&t{\Th_{0}}^{0}, \\
{(\Th_{t})_{0}}^{\a}&=&t{\Th_{0}}^{\a}+(t-t^2){\th_{0}}^{0}\wedge{\th_{0}}^{\a}, \\
{(\Th_{t})_{\b}}^{0}&=&{\Th_{\b}}^{0}+(1-t){\th_{\b}}^{0}\wedge{\th_{0}}^{0}, \\
{(\Th_{t})_{\b}}^{\a}&=&{\Th_{\b}}^{\a}+(1-t){\th_{\b}}^{0}\wedge{\th_{0}}^{\a}.
 \end{array}
\end{equation}
The $(n+1)$-st Chern form is expressed as 
\begin{equation*}
c_{n+1}(\Th )=\frac{1}{(n+1)!}\Bigl(\frac{\sqrt{-1}}{2\pi}\Bigr)^{n+1}P(\Th,\dots,\Th )
\end{equation*}
with the invariant polynomial
\begin{equation*}
P(A_{0},A_{1},\dots ,A_{n})=\sum_{\sigma ,\tau\in S_{n+1}} {\rm sgn}(\sigma\tau)\ {(A_{0})_{\sigma (0)}}^{\tau (0)}\cdots{(A_{n})_{\sigma (n)}}^{\tau (n)}.
\end{equation*}
Since $c_{n+1}(\Th_{0})=c_{n+1}(\widetilde{\Th})=0$, we have
\begin{align}\label{homotopy}
c_{n+1}(\Th) =c_{n+1}(\Th )-c_{n+1}(\widetilde{\Th})&=\int_{0}^{1}\frac{d}{dt}c_{n+1}(\Th_{t})dt \\ \notag
                                                    &=\frac{1}{n!}\Bigl(\frac{\sqrt{-1}}{2\pi}\Bigr)^{n+1}\int_{0}^{1}P(\dot{\Th}_t,\Th_{t},\dots ,\Th_{t})dt.
\end{align}
Here dots denote the differentiation with respect to $t$. We claim that 
\begin{equation}
P(\dot{\Th}_t,\Th_{t},\dots ,\Th_{t})=dP(\dot{\th}_t,\Th_{t},\dots ,\Th_{t})\label{d}.
\end{equation}
Since $P$ is an invariant polynomial, we can work with any frame. Fixing $t$, we take a frame such that ${({\th}_{t})_{i}}^{j}=0$ at a point. Then $d\dot{\th}_t=\dot{\Th}_t,\ d\Th_{t}=0$ at the point, so we get (\ref{d}). 
From (\ref{homotopy}) and (\ref{d}), we have $c_{n+1}(\Th)=d\Pi$ with 
\begin{equation}\label{Piint}
\Pi=\frac{1}{n!}\Bigl(\frac{\sqrt{-1}}{2\pi}\Bigr)^{n+1}\int_{0}^{1}P(\dot{\th}_t,\Th_{t},\dots ,\Th_{t})dt.
\end{equation}

Now we calculate 
\begin{equation}\label{var}
P(\dot{\th}_t,\Th_{t},\dots ,\Th_{t}) 
=\sum_{\sigma ,\tau\in S_{n+1}} {\rm sgn}(\sigma\tau)\ {(\dot{\th}_{t})_{\sigma (0)}}^{\tau (0)}{(\Th_{t})_{\sigma (1)}}^{\tau (1)}\cdots{(\Th_{t})_{\sigma (n)}}^{\tau (n)} 
\end{equation}
with respect to a local frame $\{W_0, W_{\a}\}$. In this frame, 
\begin{equation*}
\dot{\th}_{t}=\begin{pmatrix} {\th_{0}}^{0} & {\th_{0}}^{\a} \\
                                          0 & 0 
             \end{pmatrix}              
\end{equation*}
so ${(\dot{\th}_{t})_{\sigma (0)}}^{\tau (0)}=0$ unless $\sigma(0)=0$. We decompose the right-hand side of (\ref{var}) into
 two parts according to whether $\tau(0)=0$ or not:
\[
\sum_{\sigma ,\tau\in S_{n+1}} {\rm sgn}(\sigma\tau)\ {(\dot{\th}_{t})_{\sigma (0)}}^{\tau (0)}{(\Th_{t})_{\sigma (1)}}^{\tau (1)}\cdots{(\Th_{t})_{\sigma (n)}}^{\tau (n)}=\Pi_1+\Pi_2,
\]
where $\Pi_1$ and $\Pi_2$ are given by
\begin{align*}
\Pi_1&=\sum_{\sigma ,\tau\in S_{n}} {\rm sgn}(\sigma\tau)\ {\th_0}^0{(\Th_{t})_{\sigma (1)}}^{\tau (1)}\cdots{(\Th_{t})_{\sigma (n)}}^{\tau (n)} \\
&=\sum_{\sigma ,\tau\in S_{n}}{\rm sgn}(\sigma\tau)\ {\th_0}^0\Bigl({\Th_{\sigma (1)}}^{\tau (1)}+(1-t){\th_{\sigma(1)}}^{0}{\th_{0}}^{\tau(1)}\Bigr) \\
&\qquad\qquad\qquad\qquad\qquad\qquad \cdots\Bigl({\Th_{\sigma (n)}}^{\tau (n)}+(1-t){\th_{\sigma(n)}}^{0}{\th_{0}}^{\tau(n)}\Bigr) \\
&=\sum_{k=0}^{n}\binom{n}{k}(1-t)^{n-k}\Phi_{k}^{(0)}, \\
\Pi_2 &=-n\sum_{\sigma ,\tau\in S_{n}} {\rm sgn}(\sigma\tau)\ {\th_{0}}^{\tau (1)}{(\Th_{t})_{\sigma (1)}}^{0}{(\Th_{t})_{\sigma (2)}}^{\tau (2)}\cdots{(\Th_{t})_{\sigma (n)}}^{\tau (n)} \\
&=-n\sum_{\sigma ,\tau\in S_{n}}{\rm sgn}(\sigma\tau)\ {\th_{0}}^{\tau (1)}\Bigl({\Th_{\sigma (1)}}^{0}+(1-t){\th_{\sigma(1)}}^{0}{\th_{0}}^{0}\Bigr) \\
&\qquad\qquad\cdot\Bigl({\Th_{\sigma (2)}}^{\tau (2)}+(1-t){\th_{\sigma(2)}}^{0}{\th_{0}}^{\tau(2)}\Bigr)\cdots\Bigl({\Th_{\sigma (n)}}^{\tau (n)}+(1-t){\th_{\sigma(n)}}^{0}{\th_{0}}^{\tau(n)}\Bigr) \\
&=-\sum_{k=0}^{n-1}n\binom{n-1}{k}(1-t)^{n-k-1}\Phi_{k}^{(1)}+\sum_{k=0}^{n-1}n\binom{n-1}{k}(1-t)^{n-k}\Phi_{k}^{(0)}.
\end{align*} 
Substituting these forms into (\ref{Piint}) and computing the integration with respect to $t$, we obtain the formula of $\Pi$. 
\end{proof}

We prove that the integral of $\Pi$ over $\pa B_{\e}(p_{j})$ converges to the index of ${\rm Re}\,\xi$ at $p_j$ as 
$\e \to 0$. To relate the integration of the connection forms to the index, we use the following lemma.
\begin{lem}\label{lemma}
Let $(M, g)$ be an $(n+1)$-dimensional oriented Riemannian manifold and $\xi$ a vector field on a neighborhood $U$ of a point $p\in M$. Suppose that ${\rm Zero}(\xi)=\{p\}$ and $V$ is non-degenerate at $p$. For a linear connection $\nabla$ on $U$, we denote by  ${\omega_{i}}^{j}$ the connection 1-form with respect to a local oriented orthonormal frame $\{e_{0}=\xi/|\xi|_g, e_{1},\dots,e_{n}\}$ on $U\setminus\{p\}$. Then, 
\begin{equation}\label{integral}
\lim_{\e\to 0}\frac{1}{{\rm vol}(\Sp^{n})}\int_{\pa B_{\e}(p)}{\omega_{0}}^{1}\wedge\cdots\wedge{\omega_{0}}^{n}
={\rm Index}(p, \xi).
\end{equation}
 Here, ${\rm vol}(\Sp^{n})$ is the volume of the unit sphere in $\R^{n+1}$, and $B_{\e}(p)$ is the ball of radius $\e$ centered at $p$ with respect to $g$.
\end{lem}
This is a standard fact of the index of a vector field, but we include a proof as we use the estimate in the proof to study the integral of $\Pi$.
\begin{proof}Take normal coordinates $(x^0,\dots,x^n)$ about $p$. Then, in the coordinates, we can identify ${\pa B_{\e}(p)}$ with the sphere $\Sp^n_\e$ in $\R^{n+1}$. 

First we prove (\ref{integral}) in the case where $\nabla$ is the trivial connection $\widetilde{\nabla}$ with respect to some oriented orthonormal frame $(u_0,\dots,u_n)$ for $g$. Write $e_0=F^{i}(x)u_{i}$, and define $F_{\e}:\Sp^n_{\e}\longrightarrow \Sp^n_{\e}$ by $F_{\e}(x)=(\e F^0(x),\dots,\e F^n(x))$. Then
\begin{align*}
{\rm Index}(p, \xi)&=\frac{1}{{\rm vol}(\Sp^n_{\e})}\int_{\Sp^n_{\e}}F_{\e}^{\ast}vol_{\Sp^n_{\e}} \\
                &=\frac{1}{{\rm vol}(\Sp^n)}\int_{\Sp_{\e}^n}\sum_{i=0}^{n}(-1)^iF^idF^0\wedge\cdots\wedge\widehat{dF^i}\wedge\cdots\wedge dF^n.
\end{align*}
Let ${{\widetilde{\omega}}_i}^{\ j}$ denote the connection form of $\widetilde{\nabla}$ with respect to $(e_0,\dots,e_n)$.
If we write $e_j={a_j}^l(x)u_l$ with $SO(n+1)$-valued function $({a_j}^l)$, then ${a_0}^l=F^l$ and
\begin{equation*}
\widetilde{\nabla}e_0=dF^l\otimes u_l=\sum_{k}\Bigl(\sum_{l}dF^l\cdot{a_k}^l\Bigr)\otimes e_k
\end{equation*}
which implies 
\begin{equation*}
({{\widetilde{\omega}}_0}^{\ 0},{{\widetilde{\omega}}_0}^{\ 1},\dots,{{\widetilde{\omega}}_0}^{\ n})=(dF^0,dF^1,\dots, dF^n)
\begin{pmatrix}
F^0 & {a_1}^0 & \cdots & {a_n}^0 \\
F^1 & {a_1}^1 & \cdots & {a_n}^1 \\
\vdots & \vdots & \ddots & \vdots \\
F^n & {a_1}^n & \cdots & {a_n}^n
\end{pmatrix}.
\end{equation*}
Since the matrix in the right-hand side is in $SO(n+1)$,
\[
{{\widetilde{\omega}}_0}^{\ 1}\wedge\cdots\wedge{{\widetilde{\omega}}_0}^{\ n}=\sum_{i=0}^{n}(-1)^iF^idF^0\wedge\cdots\wedge\widehat{dF^i}\wedge\cdots\wedge dF^n.
\]
Therefore the integration in (\ref{integral}) equals the index before we take limit. 

Now we consider general cases. For a general $\nabla$, we can write as 
\begin{equation*}
\nabla=\widetilde{\nabla}+A,\ \ A\in\Omega^1({\rm End}\ TU).
\end{equation*}
Define $\d_{\e}:\Sp^n\longrightarrow \Sp_{\e}^n$ by $\d_{\e}(x)=\e x$ and set 
$\Omega_{\e}=\d_{\e}^{\ast}({\omega_{0}}^{1}\wedge\cdots\wedge{\omega_{0}}^{N}\big|_{T\Sp_{\e}^n})$. Then
\begin{equation*}
\frac{1}{{\rm vol}(\Sp^{n})}\int_{\pa B_{\e}(p)}{\omega_{0}}^{1}\wedge\cdots\wedge{\omega_{0}}^{n}
=\frac{1}{{\rm vol}(\Sp^{n})}\int_{\Sp^n}\Omega_{\e}.
\end{equation*}
We claim that 
\begin{equation}\label{est1}
\d_{\e}^{\ast}({{\widetilde{\omega}}_0}^{\ k}\big|_{T\Sp_{\e}^n})=O(1)
\end{equation}
and
\begin{equation}\label{est2}
\d_{\e}^{\ast}({A_i}^{j}\big|_{T\Sp_{\e}^n})=O(\e)
\end{equation}
in the frame $(e_0,\dots,e_{n})$. First, ${{\widetilde{\omega}}_0}^{\ k}=\textstyle\sum_{l}dF^l\cdot{a_k}^l$ and
\begin{align*}
\d_{\e}^{\ast}dF^l&=\d_{\e}^{\ast}d\Bigl(\frac{\xi^l}{|\xi|_{g}}\Bigr) \\
                      &=\d_{\e}^{\ast}\Bigl(\frac{d\xi^l}{|\xi|_g}-\frac{d(|\xi|_g^2)}{2|\xi|_g^3}\Bigr). 
\end{align*}
Since $p$ is a non-degenerate zero point, we have 
\[ \d_{\e}^{\ast}\xi^l=O(\e),\ \ \ \d_{\e}^{\ast}(1/|\xi|_g)=O(\e^{-1}),\ \ \ \d_{\e}^{\ast}d\xi^l=O(\e), \ \ \ 
\d_{\e}^{\ast}d(|\xi|_g^2)=O(\e^2). \]
Therefore, ${\d_{\e}}^{\ast}dF^l=O(1)$ so we obtain (\ref{est1}). The second estimate (\ref{est2}) is immediate since $A$ is a tensor valued 1-form. From these estimates, it follows that 
\begin{align*}
\Omega_{\e}&=\d_{\e}^{\ast}\Bigl(({\widetilde{\omega}_{0}}^{\ 1}+{A_0}^1)\wedge\cdots\wedge{(\widetilde{\omega}_{0}}^{n}+{A_0}^n)\Bigr)\big|_{T\Sp_{\e}^n} \\
&=\d_{\e}^{\ast}\Bigl({\widetilde{\omega}_{0}}^{\ 1}\wedge\cdots\wedge{\widetilde{\omega}_{0}}^{\ n}\Bigr)\big|_{T\Sp_{\e}^n}
+O(\e).
\end{align*} 
Consequently, 
\begin{align*}
\lim_{\e\to 0}\frac{1}{{\rm vol}(\Sp^{n})}\int_{\Sp^n}\Omega_{\e}&= \lim_{\e\to 0}\frac{1}{{\rm vol}(\Sp^{n})}\int_{\Sp^n}\d_{\e}^{\ast}\bigl({\widetilde{\omega}_{0}}^{\ 1}\wedge\cdots\wedge{\widetilde{\omega}_{0}}^{\ n}\bigr)\big|_{T\Sp_{\e}^n} \\
&=\lim_{\e\to 0}\frac{1}{{\rm vol}(\Sp^{n})}\int_{\Sp_{\e}^n}{\widetilde{\omega}_{0}}^{\ 1}\wedge\cdots\wedge{\widetilde{\omega}_{0}}^{\ n} \\
&={\rm Index}(p, \xi),
\end{align*} 
which completes the proof of the lemma.
\end{proof}
As the corollary to the proof above, we can show the complex version of this lemma:
\begin{cor}
Let $(X, g)$ be an $(n+1)$-dimensional hermitian manifold and $\xi$ a $(1,0)$-vector field on a neighborhood $U$ of a point $p\in X$. Suppose that ${\rm Zero}(\xi)=\{p\}$ and $V$ is non-degenerate at $p$. For a linear connection $\nabla$ on $T^{1,0}U$, we denote 
by  ${\omega_{i}}^{j}$ the connection 1-form with respect to a local orthonormal $(1,0)$-frame $\{e_{0}=\xi/|\xi|_g,e_{1},\dots,e_{n}\}$ on $U\setminus\{p\}$. Then
\begin{align*}
&\lim_{\e\to 0}\frac{1}{{\rm vol}(\Sp^{2n+1})}\int_{\pa B_{\e}(p)}\frac{(\sqrt{-1})^{n-1}}{2^n}\,{\omega_{0}}^{0}\wedge{\omega_{1}}^{0}\wedge{\omega_{0}}^{1}\wedge\cdots\wedge{\omega_{n}}^{0}\wedge{\omega_{0}}^{n} \\
&={\rm Index}(p,{\rm Re}\, \xi). 
\end{align*}
\end{cor}

Now consider the integral of the transgression form $\Pi$ over $\pa B_{\e}(p_j)$. As in the proof of Lemma \ref{lemma}, we 
can write as 
\[ \int_{\pa B_{\e}(p_j)}\Pi=\int_{\Sp^{2n+1}} \d_\e^{\ast}\Pi \]
in normal coordinates about $p_j$, and we can estimate as 
\[ \d_\e^{\ast}{\th_{i}}^{j}=O(1),\ \ \  \d_\e^{\ast}{{\Th_i}}^{j}=O(\e^2) \]
since ${\Th_{i}}^{j}$ is a tensor valued 2-form. As a result, we see that the terms in $\Pi$ which involves the curvature form do not contribute to the limit. Therefore, we have 
\begin{align*}
\ -\lim_{\e \to 0}\sum_{j=1}^{m}\int_{\partial B_{\e}(p_{j})}\Pi 
&=-\lim_{\e \to 0}\sum_{j=1}^{m}\int_{\partial B_{\e}(p_{j})}\Bigl(\frac{\sqrt{-1}}{2\pi}\Bigr)^{n+1}n!\ {\th_0}^0{\th_1}^0{\th_0}^1\cdots{\th_n}^0{\th_0}^n \\
&=-\Bigl(\frac{\sqrt{-1}}{2\pi}\Bigr)^{n+1}n!\ {\rm vol}(\Sp^{2n+1})\frac{2^n}{(\sqrt{-1})^{n-1}}\sum_{j=1}^{m}{\rm Index}(p_j,{\rm Re}\, \xi) \\
&=\sum_{j=1}^{m}{\rm Index}(p_j,{\rm Re}\, \xi). 
\end{align*}
The sum of the index of ${\rm Re}\, \xi$ gives the Euler characteristic of $X$ by the Poincar\'{e}-Hopf theorem (see e.g. \cite{M}):
\begin{thm}
Let $X$ be a compact complex manifold with boundary and $\xi$ a $(1,0)$-vector field on $X$ which has finitely many non-degenerate zero points $\{p_{1},\dots,p_{m}\}$. Suppose that ${\rm Re}\, \xi$ is pointing in the outer direction at the boundary. Then 
\begin{equation*}
\sum_{j=1}^{m}{\rm Index}(p_{j}, {\rm Re}\, \xi)=\chi(X).
\end{equation*} 
\end{thm}
Thus we have the Gauss-Bonnet formula 
\begin{equation}\label{decom}
\int_X c_{n+1}(\Th)=\chi(X)+\int_M \Pi. 
\end{equation}
\subsection{Boundary integral}Finally, we shall express the boundary term in (\ref{decom}) in terms of the pseudo-hermitian structure and the transverse curvature. 
\begin{prop}\label{connection}
With respect to an admissible coframe $\{\th^{0}=\partial\r,\th^{\a}\}$, the renormalized connection ${\th_{i}}^{j}$ is  expressed near $M$ with the Graham-Lee connection ${\varphi_{\a}}^{\b}$ as follows: 
\begin{align*} 
{\th_{\b}}^{\a}&={\varphi_{\b}}^{\a}+\sqrt{-1}r\vartheta{\d_{\b}}^{\a},  \\
{\th_{0}}^{\a}&=r\th^{\a}-\sqrt{-1}{A^{\a}}_{\ol{\g}}\th^{\ol{\g}}-r^{\a}\ol{\partial}\r \\
{\th_{\b}}^{0}&=-h_{\b\ol{\g}}\th^{\ol{\g}}+\frac{-\r r_{\b}}{1-r\r}\partial\r +\frac{\sqrt{-1}\r}{1-r\r}A_{\b\g}\th^{\g}, \\
{\th_{0}}^{0}&=-r\ol{\partial}\r +\frac{-\r r^{2}}{1-r\r}\partial\r +\frac{-\r}{1-r\r}\pa r, 
\end{align*} 
where $\vartheta=(\sqrt{-1}/2)(\ol{\pa}\r-\pa\r)$. The Bochnor tensor ${\Th_{i}}^{j}$ and the renormalized curvature ${W_{i}}^{j}$ are related as 
\begin{equation}\label{WTh}
{W_{i}}^{j}={\Th_{i}}^{j}-u_{i}\wedge\th^{j},
\end{equation}
where
\begin{align*}
u_{0}&=\frac{r^{2}}{1-r\r}\partial\r+\frac{1}{1-r\r}\partial r, \\
u_{\b}&=-\frac{\sqrt{-1}}{1-r\r}A_{\b\g}\th^{\g}+\frac{r_{\b}}{1-r\r}\partial\r .
\end{align*}
\end{prop}

\begin{proof}
Write ${\th _{i}}^{j}={{\th _{i}}^{j}}_{l}\ \th^{l}+ {{\th _{i}}^{j}}_{\ol{l}}\ \th^{\ol{l}}$. We compare the structure equations of both connections. Since $g$ is K\"{a}hler near $M$, ${\psi_{i}}^{j}$ satisfies
\begin{gather}\label{psistr}
d\th^{i}=\th^{j}\wedge{\psi_{j}}^{i} \\ \label{psimet}
{\psi _{i}}^{k}g_{k\ol{j}}+g_{i\ol{k}}{\psi _{\ol{j}}}^{\ol{k}}=dg_{i\ol{j}}
\end{gather}
with
\begin{equation}
(g_{i\ol{j}})=\begin{pmatrix}
                (1-r\r)/\r^{2} & 0 \\
                          0     &-h_{\a\ol{\b}}/\r 
             \end{pmatrix}.
\end{equation}
From (\ref{defth}) and (\ref{psistr}),
\begin{align*}
d\th^{\a}&=\th^{\b}\wedge{\psi _{\b}}^{\a}+\partial\r\wedge{\psi _{0}}^{\a} \\
         &=\th^{\b}\wedge\Bigl({\th _{\b}}^{\a}+\frac{\partial\r}{-\r}\ {\d _{\b}}^{\a}\Bigr)+\partial\r\wedge\Bigl({\th _{0}}^{\a}+\frac{1}{-\r}\th^{\a}\Bigr) \\
         &=\th^{\b}\wedge{\th _{\b}}^{\a}+\partial\r\wedge{\th _{0}}^{\a},
\end{align*}
while the structure equation (\ref{GLstr}) gives
\begin{equation*}
d\th^{\a}=\th^{\b}\wedge{\varphi_{\b}}^{\a}-\sqrt{-1}{A^{\a}}_{\ol{\g}}\partial\r\wedge\th^{\ol{\g}}-r^{\a}\partial\r\wedge\ol{\partial}\r+\frac{1}{2}rd\r\wedge\th^{\a}.
\end{equation*}
So we have 
\begin{equation}\label{coeff1}
{{\th _{0}}^{\a}}_{\ol{0}}=-r^{\a},\ \ \ {{\th _{0}}^{\a}}_{\ol{\b}}=-\sqrt{-1}{A^{\a}}_{\ol{\b}}.
\end{equation}
Similarly, comparing the equations
$$
-\partial\ol{\partial}\r=d\th^{0}=\th^{\b}\wedge{\psi_{\b}}^{0}+\partial\r\wedge{\psi _{0}}^{0} 
       =\th^{\b}\wedge{\th_{\b}}^{0}+\partial\r\wedge{\th_{0}}^{0}
       $$
       and
       $$
\partial\ol{\partial}\r =h_{\a\ol{\b}} \th^{\a}\wedge\th^{\ol{\b}}+r\partial\r\wedge\ol{\partial}\r,
$$
we get
\begin{equation}\label{coeff2}
{{\th _{0}}^{0}}_{\ol{0}} =-r,\ \ \ {{\th _{0}}^{0}}_{\ol{\g}}=0,\ \ \ {{\th _{0}}^{0}}_{\g}={{\th _{\g}}^{0}}_{0},\ \ \ 
{{\th _{\a}}^{0}}_{\ol{0}} =0,\ \ \ {{\th _{\a}}^{0}}_{\ol{\g}}=-h_{\a\ol{\g}}.
\end{equation}
The equation (\ref{psimet}) for $(i, \ol{j})=(\a, \ol{0})$ is equivalent to
\begin{equation}\label{met}
\frac{1-r\r}{\r^{2}}{{\th _{\a}}^{0}}+\frac{h_{\a\ol{\b}}}{-\r}\Bigl( {{\th _{\ol{0}}}^{\ol{\b}}}+\frac{1}{-\r}\th^{\ol{\b}}\Bigr)=0.
\end{equation}
From (\ref{coeff1}), (\ref{coeff2}), and (\ref{met}), we have
\begin{equation*}
{{\th _{\a}}^{0}}_{\g}=\frac{\sqrt{-1}\r}{1-r\r}A_{\a\g},\ \ \ {{\th _{\a}}^{0}}_{0}=\frac{-\r}{1-r\r}r_{\a},\ \ \ 
{{\th _{0}}^{\b}}_{\g}=r{\d _{\g}}^{\b},\ \ \ {{\th _{0}}^{\b}}_{0}=0.
\end{equation*}
The equation (\ref{psimet}) for $(i, \ol{j})=(0, \ol{0}), (\a, \ol{\b})$ and (\ref{coeff2}) give  
\begin{equation*}
{{\th _{0}}^{0}}_{0}=\frac{-\r (r^{2}+r_{0})}{1-r\r},\ \ \ {\th _{\a}}^{\g}h_{\g\ol{\b}}+h_{\a\ol{\g}}{\th _{\ol{\b}}}^{\ol{\g}}=dh_{\a\ol{\b}}.
\end{equation*}
Finally, using (\ref{coeff1}), we obtain
\begin{align*}
d\th^{\a}&=\th^{\b}\wedge{\th _{\b}}^{\a}+\pa\r\wedge{\th _{0}}^{\a} \\
         &=\th^{\b}\wedge\Bigl( {\th _{\b}}^{\a}-\sqrt{-1}r\vartheta{\d _{\b}}^{\a}\Bigr) -\sqrt{-1}{A^{\a}}_{\ol{\g}}\pa\r\wedge\th^{\ol{\g}}-r^{\a}\pa\r\wedge\ol{\pa}\r+\frac{r}{2}d\r\wedge\th^{\a}.
\end{align*}
This implies that ${\th _{\b}}^{\a}-\sqrt{-1}r\vartheta{\d _{\b}}^{\a}$ satisfies the characterization of the Graham-Lee  connection ${\varphi_{\b}}^{\a}$. Now it is straightforward to calculate ${W_{i}}^{j}$:
\begin{align*} 
{W_{i}}^{j}&={\Psi_{i}}^{j}+{K_{i}}^{j} \\
           &={\Th_{i}}^{j}-d{Y_{i}}^{j}-{Y_{i}}^{k}\wedge{Y_{k}}^{j}+{Y_{i}}^{k}\wedge{\th_{k}}^{j}+{\th_{i}}^{k}\wedge{Y_{k}}^{j}+{K_{i}}^{j}
\end{align*}
and 
\begin{align*}
d{Y_{i}}^{j}+{Y_{i}}^{k}\wedge{Y_{k}}^{j}&=(g_{k\ol{l}}\th^k\wedge\th^{\ol{l}}){\d_{i}}^{j}-{\d_{i}}^{0}\frac{\ol{\partial}\r\wedge\th^{j}}{\r^{2}}+{\d_{i}}^{0}\frac{d\th^{j}}{\r} \\
{Y_{\b}}^{k}\wedge{\th_{k}}^{\a}+{\th_{\b}}^{k}\wedge{Y_{k}}^{\a}&=\frac{1}{\r}{\th _{\b}}^{0}\wedge\th^{\a} \\
{Y_{0}}^{k}\wedge{\th_{k}}^{0}+{\th_{0}}^{k}\wedge{Y_{k}}^{0}&=\frac{1}{\r}\th^{\g}\wedge{\th_{\g}}^{0}      \\
{Y_{\b}}^{k}\wedge{\th_{k}}^{0}+{\th_{\b}}^{k}\wedge{Y_{k}}^{0}&=-\frac{1}{\r}\partial\r\wedge{\th_{\g}}^{0}  \\
{Y_{0}}^{k}\wedge{\th_{k}}^{\a}+{\th_{0}}^{k}\wedge{Y_{k}}^{\a}&=\frac{1}{\r}\th^{\g}\wedge{\th_{\g}}^{\a}+\frac{1}{\r}
\partial\r\wedge{\th_{0}}^{\a}+\frac{1}{\r}{\th_{0}}^{0}\wedge\th^{\a} \\
{K_{i}}^{j}&=(g_{k\ol{l}}\th^k\wedge\th^{\ol{l}}){\d_{i}}^{j}+g_{i\ol{l}}\th^{j}\wedge\th^{\ol{l}}
\end{align*}
from which (\ref{WTh}) follows.
\end{proof}
From the proposition above, we can express the renormalized connection and curvature at the boundary with the transverse curvature, the TW curvature and torsion, and their covariant derivatives:
\begin{cor}\label{conn-boundary}
With respect to an admissible coframe $\{\th^{0}=\partial\r,\th^{\a}\}$, the renormalized connection and curvature satisfy 
\begin{align*}
{\th_{\b}}^\a|_{TM} &={\omega_\b}^\a+\sqrt{-1} r\th{\delta_\b}^\a, \\
{\th_{0}}^\a|_{TM}  &=r\th^\a-\sqrt{-1}{A^\a}_{\ol{\g}}\th^{\ol{\g}}+\sqrt{-1}r^\a\th, \\
{\th_{\b}}^0|_{TM}  &=-h_{\a\ol{\g}}\th^{\ol{\g}}, \\
{\th_{0}}^0|_{TM}   &=\sqrt{-1}r\th, \\
{\Th_{\b}}^\a|_{TM} &=({{R_\b}^\a}_{\g\ol{\d}}-r{\d_\b}^\a h_{\g\ol{\d}}-r{\d_\g}^\a h_{\b\ol{\d}})\th^\g\wedge\th^{\ol{\d}}+({A_{\b\g,}}^\a+\sqrt{-1}r_\g{\d_\b}^\a )\th^\g\wedge\th \\
&\quad-({A^\a}_{\ol{\g},\b}-\sqrt{-1}r_{\ol{\g}}{\d_\b}^\a-\sqrt{-1}r^\a h_{\b\ol{\g}})\th^{\ol{\g}}\wedge\th-\sqrt{-1}A_{\b\g}\th^\g\wedge\th^\a, \notag \\
{\Th_{\b}}^0|_{TM}  &=A_{\b\g}\th^\g\wedge\th.
\end{align*}
\end{cor}
When we substitute the above formulas to $\Pi$, we can neglect terms involving $\th$ in 
${\th_{0}}^{\a}$ and ${\Th_{\b}}^{\a}$ since each $\Phi_{k}^{(i)}$ contains ${\th_{0}}^{0}$ or ${\Th_{\b}}^{0}$. Consequently, $\Pi|_{TM}$ contains no covariant derivatives and we obtain the boundary term $\mu(M, \r)$ of the form stated in Theorem \ref{first-mainthm}.
\section{Proof of Theorem \ref{second-mainthm}}\label{section-pe}
  \subsection{Pseudo-Einstein structures} 
We start with a review of some notions concerned. Let $M$ be a $(2n+1)$-dimensional strictly pseudoconvex CR manifold. The {\it CR canonical bundle} $K_M$ is defined by $\bigwedge^{n+1}(T^{0,1}M)^{\perp}\subset \bigwedge^{n+1}(\C\otimes T^{\ast}M)$. When a local section $\zeta$ of $K_M$ satisfies 
\begin{equation*}
\th\wedge (d\th)^{n}=(\sqrt{-1})^{n^2}n!\,\th\wedge(T\lrcorner \z)\wedge(T\lrcorner \ol{\z})
\end{equation*}
for a contact form $\th$, we say that $\th$ is {\it volume normalized} with respect to $\zeta$. 
We remark that $\th$ is volume normalized if and only if $\zeta$ is locally expressed in the form
\begin{equation*}
\z =e^{\sqrt{-1}\g}\th\wedge\th^{1}\wedge\cdots\wedge\th^{n}
\end{equation*}
with a real function $\g$, for an admissible coframe satisfying $d\th =\sqrt{-1}\sum\th^{\a}\wedge\th^{\ol{\a}}$. 
\begin{dfn}
A contact form $\th$ is {\em pseudo-Einstein} if, in a neighborhood of each point $p\in M$, there exists a closed section of $K_M$ with respect to which $\th$ is volume normalized.
\end{dfn}
It is shown in \cite{Lee2} that, when $n\ge 2$, a contact form is pseudo-Einstein if and 
only if the TW Ricci curvature satisfies 
\begin{equation}
{\rm Ric}_{\a\ol{\b}}=\frac{1}{n}\scal\  h_{\a\ol{\b}}.\label{pseudo}
\end{equation}
This equation is analogous to the Einstein equation in Riemannian geometry, but (\ref{pseudo}) does not imply that the 
scalar curvature is constant. When $n=1$, the equation (\ref{pseudo}) is always satisfied, and it is shown in \cite{Hi1} that a contact form is pseudo-Einstein if and only if $(\scal)_{1}-\sqrt{-1}{A_{11,}}^{1}=0$. 

The set of pseudo-Einstein contact forms is parametrized by CR pluriharmonic functions. This fact is shown in \cite{Lee2} 
for $n\ge 2$, but the proof is also valid for $n=1$.
\begin{prop}[\cite{Lee2}]
If $\th$ is pseudo-Einstein, then $\widehat{\th}=e^{\U}\th$ is pseudo-Einstein if and only if $\U$ is CR plurihamonic.
\end{prop}
\subsection{Complex Monge-Amp\`{e}re equation}Let $\r$ be a defining function of an $(n+1)$-dimensional pseudoconvex manifold $X$ with the boundary $M$. Take local holomorphic coordinates $(z^0,\dots,z^n)$. We denote partial derivatives by subscripts, e.g. $\pa f/\pa z^i=f_i$. We set 
\begin{equation*}
\J[\r]= \det \begin{pmatrix}
                    \r        & \r_{\ol{j}} \\
                    \r_{i}    & \r_{i\ol{j}} 
              \end{pmatrix}      
\end{equation*} 
and call $\J$ {\it the Monge-Amp\`{e}re operator}. By the column transformation,
\begin{equation}\label{monge}
\J[\r]=\det \begin{pmatrix}
                    \r        & 0                                            \\
                    \r_{i}    & \r_{i\ol{j}}-\frac{\r _{i}\r_{\ol{j}}}{\r} 
              \end{pmatrix} 
     =-(-\r)^{n+2}\det\Bigl(\frac{\pa^2}{\pa z^i\pa z^{\ol{j}}}\log \frac{1}{-\r}\Bigr).
\end{equation}
From this expression, one can see that
\begin{equation}\label{invprop}
\J[e^{f}\r]=e^{(n+2)f}\J[\r]\quad\text{if}\quad \pa\ol{\pa}f=0. 
\end{equation}
Also, we have the transformation law for the change of local coordinates: Let $w=F(z)$ be another holomorphic coordinate system. We distinguish the Monge-Amp\`{e}re operators in two coordinate systems by the corresponding subscripts. Then, 
\begin{equation*}
\J_{z}[\r]=|\det F^{\prime}|^{2}\J_{w}[\r],
\end{equation*}
where $F^{\prime}$ denotes the holomorphic Jacobian. Combining the above two formulas, we have 
\begin{equation*}
\J_{z}[|\det F^{\prime}|^{-\frac{2}{n+2}}\r] = \J_{w}[\r].
\end{equation*}
Suppose that $g=\pa\ol{\pa}\log(-1/\r)$ defines a K\"{a}hler metric. From
 (\ref{monge}), 
\begin{equation}\label{MongeEin}
-\pa\ol{\pa}\log(-\J[\r])={\rm Ric}(g)+(n+2)g,
\end{equation}
so if $\J[\r]=-e^{f}$ 
for some pluriharmonic function $f$, then $g$ is a K\"{a}hler-Einstein metric. By (\ref{invprop}), to find such a defining function is equivalent to solving the following {\it Monge-Amp\`{e}re equation}:
\begin{align*}
\J[\r]&=-1,\ \r<0\ \ \text{in}\ X \\
\r&=0 \ \ \ \ \ \ \ \ \ \ \ \ \text{on}\ M.
\end{align*}
For a strictly pseudoconvex domain in $\C^{n+1}$, Cheng and Yau \cite{CY} established the existence and uniqueness of the solution to the Monge-Amp\`{e}re equation. However, we do not need the exact solution for our purpose. We use the following Fefferman's approximate solutions:
\begin{thm}[\cite{Fe2}]
For a strictly pseudoconvex domain $X$ in $\C^{n+1}$, there exists a defining function $\r\in C^{\infty}(\ol{X})$ such that 
\begin{equation}\label{approxMonge}
\J[\r]=-1+O(\r ^{n+2}).
\end{equation}
If another defining function $\widetilde{\r}$ satisfies the same equation, then
\begin{equation*}
\r -\widetilde{\r}=O(\r^{n+3}).
\end{equation*}
\end{thm}
We can take Fefferman's defining function in local coordinates around each point on the boundary, but we need a global potential function to apply Theorem \ref{first-mainthm}.
So we assume that $X$ admits a {\it global approximate solution} to the Monge-Amp\`{e}re equation, in the sense of \cite{HPT}:
\begin{dfn}\label{global}
A function $\r \in C^{\infty}(\ol{X})$ is called a global approximate solution to the Monge-Amp\`{e}re equation if for any
$p\in M$, there is a holomorphic coordinate system around $p$ such that $\r$ is a Fefferman's approximate solution 
in the chosen coordinates.
\end{dfn}
\noindent This notion is closely related to the pseudo-Einstein structures on the boundary:
\begin{thm}[\cite{HPT}]\label{HPTthm}
There is a global approximate solution
$\rho$ to the Monge-Amp\`{e}re equation on $X$ if and only if 
$M$ carries a pseudo-Einstein contact form. In this case, the contact form $(\sqrt{-1}/2)(\ol{\pa}\r-\pa\r)|_{TM}$ induced by a global approximate solution $\r$ is pseudo-Einstein. 
\end{thm}
\subsection{Approximate Einstein equation}When $\r$ is an approximate solution to the Monge-Amp\`{e}re equation, the metric 
$g=\pa\ol{\pa}\log(-1/\r)$ satisfies the approximate Einstein equation:
\begin{equation}\label{approx}
{\rm Ric}(g)+(n+2)g=-\pa\ol{\pa}\log(-\J[\r])=\pa\ol{\pa}\,O(\r^{n+2}),
\end{equation}
where $\pa\ol{\pa}\,O(\r^{n+2})$ stands for a term of the form $\pa\ol{\pa}(\r^{n+2}\phi)$ for some function $\phi\in 
C^{\infty}(\ol{X})$. In order to derive a formula for the boundary value of the transverse curvature from the Einstein equation, let us calculate the trace of the renormalized curvature in two ways: Firstly, taking traces of (\ref{defW}) and (\ref{WTh}), and using $u_i\wedge\th^i=0$, we have  
\begin{equation*}
{\Th_{i}}^{i}={W_{i}}^{i}={\Psi_{i}}^{i}+{K_{i}}^{i}={\rm Ric}(g)+(n+2)g,
\end{equation*}
so that, by (\ref{approx}),
\begin{equation}\label{Th1}
{\Th_{i}}^{i}=\pa\ol{\pa}\,O(\r^{n+2}).
\end{equation}
Secondly, by Propositions \ref{GLcurv} and \ref{connection}, we have 
\begin{equation}\label{Th2}
{\Th_{i}}^{i}=d{\th_{\g}}^{\g}+d{\th_{0}}^{0},
\end{equation}
where 
\begin{align*}
d{\th_{\g}}^{\g}&=d{\varphi_{\g}}^{\g}+\sqrt{-1}ndr\wedge\th +\sqrt{-1}nrd\th, \\
d{\varphi_{\g}}^{\g}&={\rm Ric}_{\g\ol{\mu}}\th^{\g}\wedge\th^{\ol{\mu}}+\sqrt{-1}{A_{\b\g,}}^{\b}\th^{\g}\wedge\ol{\partial}\r +\sqrt{-1}{A_{\ol{\b}\ol{\g},}}^{\ol{\b}}\th^{\ol{\g}}\wedge\partial\r  \\
&\quad+d\r\wedge\frac{n+2}{2}\Bigl(r_{\g}\th^{\g}-r_{\ol{\g}}\th^{\ol{\g}}\Bigr)+\frac{1}{2}\Bigl(\Delta_{b}r-2|A|^2\Bigr)\partial\r\wedge\ol{\partial}\r, \notag \\
d{\th_{0}}^{0}&=-dr\wedge\ol{\partial}\r-r\partial\ol{\partial}\r +\frac{r^2}{1-r\r}\partial\r\wedge\ol{\partial}\r +\frac{-2r\r}{1-r\r}dr\wedge\partial\r \\
&\quad+\frac{-r^2\r}{(1-r\r)^{2}}d(r\r)\wedge\partial\r+\frac{\r r^2}{1-r\r}\partial\ol{\partial}\r-\frac{1}{1-r\r}d\r\wedge\partial r \notag \\
&\quad+\frac{-\r}{(1-r\r)^2}d(r\r)\wedge\partial r+\frac{\r}{1-r\r}\partial\ol{\partial}r. \notag
\end{align*}
Comparing the coefficients of $\pa\r\wedge\ol{\pa}\r$, $\th^{\g}\wedge\ol{\partial}\r$, $\th^{\a}\wedge\th^{\ol{\b}}$ in 
(\ref{Th1}) and (\ref{Th2}), we obtain
\begin{multline}
-nr_{N}+\frac{1}{2}(\Delta_{b}r-2|A|^2)-nr^2+\frac{\r}{1-r\r}r_{0\ol{0}}+\frac{2r^3\r}{1-r\r} \\
+\frac{r^2\r^2}{(1-r\r)^2}r_{\ol{0}}+\frac{r^3\r}{(1-r\r)^2}+\frac{\r^2}{(1-r\r)^2}r_{0}r_{\ol{0}}+\frac{r\r}{(1-r\r)^2}r_{0} \\
+\frac{4r\r}{1-r\r}r_{N}+\frac{\r}{1-r\r}r_{\g}r^{\g}=O(\r^n),\label{Ein1} 
\end{multline}
\begin{multline}
\sqrt{-1}{A_{\g\b,}}^{\b}-(n+2)r_{\g}+\frac{r_{\g}}{(1-r\r)^2}  \\
+\frac{\r^2}{(1-r\r)^2}r_{\ol{0}}r_{\g}+\frac{\r}{1-r\r}(r_{\ol{0}\g}-\sqrt{-1}r_{\ol{\mu}}{A^{\ol{\mu}}}_{\g})=O(\r^{n+1}),\label{Ein2} 
\end{multline}
\begin{multline}
{\rm Ric}_{\a\ol{\b}}-(n+1)rh_{\a\ol{\b}}+\frac{\r^2}{(1-r\r)^2}r_{\a}r_{\ol{\b}}\\
+\frac{\r r^2}{1-r\r}h_{\a\ol{\b}}
+\frac{\r}{1-r\r}(r_0h_{\a\ol{\b}}+r_{\a\ol{\b}})=O(\r^{n+1}).\label{Ein3}
\end{multline}
Setting $\r=0$ in the above equations, we have
\begin{equation}
{r_{N}|}_{M}=\frac{1}{2n}\Delta_{b}r-r^2-\frac{1}{n}|A|^2,\label{rN}
\end{equation}
\begin{equation}
\sqrt{-1}{A_{\g\b,}}^{\b}-(n+1)r_{\g}|_M=0,\label{Bianchi2}
\end{equation}
\begin{equation}
{\rm Ric}_{\a\ol{\b}}=(n+1)r|_Mh_{\a\ol{\b}}.\label{pe}
\end{equation}
For $n>1$, the equation (\ref{pe}) implies that the induced pseudo-hermitian structure at the boundary is pseudo-Einstein and
\begin{equation}
r|_{M}=\frac{\scal}{n(n+1)}.\label{rM}
\end{equation}
For $n=1$ the pseudo-Einstein condition follows from (\ref{Bianchi2}) and (\ref{rM}). Differentiating (\ref{Ein1}) 
in the $N$ direction $k$ times and setting $\r=0$ gives expressions of $N^{k+1}r|_M$ for $1\le k\le n-1$. 
When $k=n$, the coefficient of $N^{n+1}r|_M$ in the left-hand side becomes 0 and the right-hand side is a constant multiple of the Graham's obstruction function, which is known as a CR invariant (see \cite{G1}). Therefore, in this case we obtain an expression of the obstruction function in terms of the pseudo-hermitian structure. 

\quad \\

From (\ref{rM}) and Corollary \ref{conn-boundary}, we can express $\mu(M,\r)$ only with pseudo-hermitian structure.  
For $n=1$, we have
\begin{equation*}
\int_{M^3}\Pi=\frac{1}{4\pi^2}\int_{M}\Bigl(|A|^2-\frac{1}{4}(\scal)^2\Bigr)\th\wedge d\th.
\end{equation*}
This agrees with the Burns-Epstein invariant and (\ref{decom}) recovers the formula 
in \cite{BE2}. For $n=2$, we have 
\begin{equation}\label{five-dim-mu}
\int_{M^5}\Pi 
=\frac{1}{16\pi^3}\int_{M}\Bigl(\frac{1}{54}(\scal)^3-\frac{1}{12}|R|^2\scal+R_{\a\ol{\b}\g\ol{\d}}A^{\a\g}A^{\ol{\b}\ol{\d}}\Bigr)\th\wedge (d\th)^2.
\end{equation}
\subsection{CR invariance}
It remains to show that $\mu(M,\r)$ gives a CR invariant, which we denote by $\mu(M)$. 
Recall that $\int_M \Pi_\th$ is defined for a CR manifold with a pseudo-Einstein contact form $\th$. We shall prove that this integral is invariant under a change of $\th$. 
\begin{thm}
 If $\th$ and $\widehat{\th}$ are pseudo-Einstein contact forms on $M$, then 
\begin{equation*}
\int_M \Pi_\th=\int_M \Pi_{\widehat{\th}}.
\end{equation*}
\end{thm} 
\begin{proof}
In the case $n=1$, the integral is the Burns-Epstein invariant so it is invariant under the change of $\th$. 
In higher dimensions, we can realize $M$ as the boundary of a strictly pseudoconvex domain $X$ in a K\"{a}hler manifold 
(see \cite[Theorem 8.1]{Le}). We may assume that $M$ is connected. Let $\r$, $\widehat{\r}$ be global approximate solutions to the Monge-Amp\`{e}re equation such that $(\sqrt{-1}/2)(\ol{\pa}\r-\pa\r)|_{TM}=\th$ and $(\sqrt{-1}/2)(\ol{\pa}\widehat{\r}-\pa\widehat{\r})|_{TM}=\widehat{\th}$. We can write as $\widehat{\th}=e^\U\th$ with a 
CR pluriharmonic function $\U$. Since $X$ is K\"{a}hler and $M$ is connected, $\U$ extends to a pluriharmonic function on $X$ (see \cite[Theorem 7.1]{Hi2}). We set $\r_t=e^{t\U}\r$ $(t\in[0,1])$, and write the corresponding renormalized connection and curvature as ${(\th_{t})_{i}}^{j}$ and ${(\Th_{t})_{i}}^{j}$ respectively. By the renormalized Gauss-Bonnet formula, it 
suffices to show that 
\begin{equation}\label{van-anomaly}
\frac{d}{dt}\Big|_{t=0}\int_{X}c_{n+1}(\Th_{t})=0.
\end{equation}
From (\ref{defth}), we have 
\begin{equation*}
{(\th_{t})_i}^j={\th_i}^j+t({\d_i}^j\U_k+{\d_k}^j\U_i)\th^k
\end{equation*}
so that 
\begin{equation*}
{\dot{\th}_i}^{\ j}=({\d_i}^j\U_k+{\d_k}^j\U_i)\ \th^k.
\end{equation*}
Here the dot denotes the differentiation in $t$ at $t=0$. Set 
\begin{equation*}
\a=\sum_{\sigma,\tau}{\rm sign}(\sigma\tau)\ {\dot{\th}_{\sigma(0)}}^{\quad\tau(0)}{\Th_{\sigma(1)}}^{\tau(1)}\cdots{\Th_{\sigma(n)}}^{\tau(n)}.
\end{equation*}
Then, calculating with a $(1,0)$-frame for which ${\th_i}^j=0$ at a point, we have 
\begin{align*}
d\a &=\sum_{\sigma,\tau}{\rm sign}(\sigma\tau)\ {\dot{\Th}_{\sigma(0)}}^{\quad\tau(0)}{\Th_{\sigma(1)}}^{\tau(1)}\cdots{\Th_{\sigma(n)}}^{\tau(n)} \\
&=\frac{1}{n+1}\cdot(n+1)!\Bigl(\frac{\sqrt{-1}}{2\pi}\Bigr)^{-(n+1)}\frac{d}{dt}\Big|_{t=0}c_{n+1}(\Th_{t})
\end{align*}
and so 
\begin{equation}\label{anomaly}
\frac{d}{dt}\Big|_{t=0}\int_{X}c_{n+1}(\Th_{t})=\frac{1}{n!}\Bigl(\frac{\sqrt{-1}}{2\pi}\Bigr)^{n+1}\int_{X}d\a.
\end{equation}
Let us prove that $\a$ can be expressed as $\pa\U\wedge Q(\Th)$ near $M$ for some ${\rm GL}(n+1,\R)$-invariant polynomial $Q$.
Put 
\begin{equation*}
P(A_{0},A_{1},\cdots ,A_{n})= 
\sum_{\sigma ,\tau\in S_{n+1}} {\rm sgn}(\sigma\tau)\ {(A_{0})_{\sigma (0)}}^{\tau (0)}\cdots{(A_{n})_{\sigma (n)}}^{\tau (n)}
\end{equation*}
so that
\begin{align}
\a &=P(({\d_i}^j\U_k+{\d_k}^j\U_i)\th^k,\Th,\cdots,\Th) \label{ALPHA} \\ 
&=P(\pa\U{\d_i}^j,\Th,\cdots,\Th)+P(\U_i\th^j,\Th,\cdots,\Th). \notag
\end{align}
Since $P$ is an invariant polynomial, it can be written as   
\begin{align*}
P(A_{0},\cdots,A_{n})=\sum_{\stackrel{i_0+\cdots +i_l=n}{\sigma\in S_{n+1}}}a_{i_0\cdots i_l}&{\rm tr}\Bigl(A_{\sigma(0)}\cdots A_{\sigma(i_0)}\Bigr){\rm tr}\Bigl(A_{\sigma(i_0+1)} 
\cdots A_{\sigma(i_0+i_1)}\Bigr) \\
&\qquad\cdots{\rm tr}\Bigl(A_{\sigma(i_{0}+\cdots +i_{l-1}+1)}\cdots A_{\sigma(n)}\Bigr). 
\end{align*}
The first term in the last line of (\ref{ALPHA}) can be written as $\pa\U\wedge Q^{\prime}(\Th)$ with an invariant polynomial $Q^{\prime}$. By the definition, the renormalized connection satisfies $d\th^i=\th^j\wedge{\th_j}^i$ and 
$\th^j\wedge{\Th_{j}}^i=0$ near $M$. Therefore, in $P(\U_i\th^j,\Th,\cdots,\Th)$, terms which do not involve ${\rm tr}(\U_i\th^j)(=\pa\U)$ must vanish near the boundary. So this is also of the required form. 

Now by Stokes' theorem we can replace $\a$ in (\ref{anomaly}) by $\pa\U\wedge Q(\Th)$. Since $\U$ is pluriharmonic in $X$ 
and $Q$ is an invariant polynomial, we have $d(\pa\U\wedge Q(\Th))=0$, from which (\ref{van-anomaly}) follows.
\end{proof} 
\section{Examples}
\subsection{Tube domains}
Let $L$ be a negative line bundle over a $2$-dimensional compact complex manifold $Y$, and $h$ a hermitian metric on $L$ such that $\underline{g}=\pa\ol{\pa}\log h>0$. We assume that $\underline{g}$ is a K\"{a}hler-Einstein metric. We consider a tube domain $X=\{v\in L\,|\, h(v,v)<1\}$ and its boundary $M=\pa X$. A defining function $\r=\log h(v,v)$ satisfies $\pa\ol{\pa}\r=\underline{g}_{\a\ol{\b}}\th^\a\wedge\th^{\ol{\b}}$, so $M$ is strictly pseudoconvex and the Levi form with respect to a contact form $\th=(\sqrt{-1}/2)(\ol{\pa}\r-\pa\r)|_{TM}$ is $\underline{g}_{\a\ol{\b}}$. It also follows that the TW connection is  given by the Levi-Civita connection ${\omega_{\b}}^{\a}$ of $\underline{g}$ and the TW torsion vanishes: $A_{\a\b}\equiv 0$. Since $\underline{g}$ is Einstein, $\th$ is pseudo-Einstein. We shall prove: 
\begin{prop}
Let L, Y, \underline{g}, M be as above. Then 
\[ \mu(M)=\frac{\sigma}{36}\Bigl(\chi(Y)-\frac{1}{8\pi^2}\int_Y |{\rm Weyl}|^2vol_{\underline{g}}\Bigr),  \]
 where $\sigma$ and {\rm Weyl} are respectively the scalar and the Weyl curvatures of \underline{g} as a Riemannian metric.
\end{prop}
\begin{proof}
Since $A_{\a\b}=0$ and $\scal$ is constant, the formula (\ref{five-dim-mu}) gives  
\begin{equation}\label{tube-mu}
\mu(M)=\frac{\scal}{96\pi^3}\int_M \Bigl(\frac{1}{9}\scal^2- \frac{1}{2}|R|^2\Bigr)\th\wedge(d\th)^2.
\end{equation}
Taking fiber coordinates $\z=re^{\sqrt{-1}t}$, we compute as
\begin{align}\label{volume-tube}
\th\wedge(d\th)^2 &=\text{Im}(\pa\r|_M)\wedge (\sqrt{-1}\underline{g}_{\a\ol{\b}}\th^\a\wedge\th^{\ol{\b}})^2 \\
                  &=\text{Im}\Bigl(\pa\log h+\frac{d\z}{\z}\Bigr)\Big|_M \wedge 2\,vol_{\underline{g}} \notag \\
                  &=2dt\wedge vol_{\underline{g}}. \notag
\end{align}
We regard $\underline{g}$ as a Riemannian metric, and denote the Riemannian and the Weyl curvatures by $(\text{Rm})_{IJKL}$ and $(\text{Weyl})_{IJKL}$ respectively. Here the indices run through values $1, 2, \ol{1}, \ol{2}$. Our sign convention is such that $(\text{Rm})_{\a\ol{\b}\g\ol{\d}}=R_{\a\ol{\b}\g\ol{\d}}$. Then since $\underline{g}$ is Einstein, we have 
\[(\text{Rm})_{IJKL}=(\text{Weyl})_{IJKL}+\frac{\sigma}{12}(\underline{g}_{IL}\underline{g}_{JK}-\underline{g}_{IK}\underline{g}_{JL}),\]
where $\sigma=2\,\scal$ is the Riemannian scalar curvature of $\underline{g}$. Hence we have  
\begin{equation}\label{weyl}
|{\rm Weyl}|^2=4|R|^2 -\frac{1}{6}\sigma^2.
\end{equation}
From (\ref{tube-mu}), (\ref{volume-tube}) and (\ref{weyl}), we obtain 
\[\mu(M)=\frac{\sigma}{36}\cdot\frac{1}{32\pi^2}\int_{Y}\Bigl(\frac{1}{6}\sigma^2-3|\text{Weyl}|^2\Bigr)vol_{\underline{g}}.\]
On the other hand, the Gauss-Bonnet formula on an Einstein 4-manifold implies 
\[\frac{1}{32\pi^2}\int_{Y}\Bigl(\frac{1}{6}\sigma^2+|\text{Weyl}|^2\Bigr)vol_{\underline{g}}=\chi(Y).\]
Consequently we have 
\[ \mu(M)=\frac{\sigma}{36}\Bigl(\chi(Y)-\frac{1}{8\pi^2}\int_Y |{\rm Weyl}|^2vol_{\underline{g}}\Bigr).  \]
\end{proof}
\subsection{Reinhardt domains}As a second example, we consider a family of Reinhardt domains $\Omega_r=\{(w^0, w^1, w^2)\in\C^3\ |\ \sum(\log|w^i|)^2<r^2\}$. These domains are strictly pseudoconvex and the formula of $\mu(\pa\Omega_r)$ is given by the following proposition.
\begin{prop}
Let $\Omega_r\subset \C^3$ be as above. Then 
\[\mu(\pa\Omega_r)=-\frac{20\pi}{27}\frac{1}{r^3}.\]
Moreover, $\Omega_r$ and $\Omega_{r^\prime}$ are biholomorphic if and only if $r=r^\prime$.
\end{prop}
\begin{proof}
By taking $\log$ of each coordinate, $\Omega_r$ is mapped biholomorphically to $\Omega^{\prime}_r=\{(z^i=x^i+\sqrt{-1}y^i)\ |\ \sum(x^i)^2<r^2,\, y^i\in\R/{2\pi\mathbb{Z}}\}\subset \C^3/\mathbb{Z}^3$. We take $\r=2(\sum(x^i)^2-r^2)$ as a defining function of $\Omega^{\prime}_r$ and compute $\mu(\pa \Omega^{\prime}_r)$. Since $O(3)\times\mathbb{T}^3$ acts transitively on $\pa\Omega^{\prime}_r$ as local CR diffeomorphisms and the contact form $\th=(\sqrt{-1}/2)(\ol{\pa}\r-\pa\r)$ is invariant under this action, it suffices to compute $\Pi_{\th}$ at a point $p=(r,0,0)\in\pa\Omega^{\prime}_r$. In a neighborhood of $p$, we can write as 
\[ \pa\ol{\pa}\r=h_{\a\ol{\b}}\th^{\a}\wedge\th^{\ol{\b}}+\kappa\pa\r\wedge\ol{\pa}\r,\]
where
\begin{align*}
h_{\a\ol{\b}}&=\d_{\a\ol{\b}}+(x^0)^{-2}x^{\a}x^{\b},\ \ \ \th^{\a}=dz^{\a}-\frac{1}{2}(x^0)^{-2}\sum_{\g}h^{\a\ol{\g}}x^{\g}\pa\r, \\
\kappa&=\frac{1}{4}(x^0)^{-2}-\frac{1}{4}(x^0)^{-4}\sum_{\mu, \g}h^{\mu\ol{\g}}x^{\mu}x^{\g}.
\end{align*}
We set $x^{\prime}=(x^1,x^2)$. Then, since $(x^0)^{-2}=r^{-2}+O(|x^{\prime}|^2)$ on $\pa\Omega^{\prime}_r$,  
\begin{align*}
h_{\a\ol{\b}}&=\d_{\a\ol{\b}}+\frac{x^\a x^\b}{r^2}+O(|x^\prime|^4), \\ 
h^{\a\ol{\b}}&=\d^{\a\ol{\b}}-\frac{x^\a x^\b}{r^2}+O(|x^\prime|^4), \\
\th^\a|_{T\pa\Omega^\prime_r} &=dz^\a-\frac{\sqrt{-1}}{2}\,\frac{x^\a}{r^2}\th+O(|x^\prime|^3).
\end{align*}
Differentiating the last equation, we have 
\[
d\th^\a|_{T\pa\Omega^\prime_r}=\th^\b\wedge\Bigl(-\frac{\sqrt{-1}}{4r^2}\th{\d_\b}^\a+\frac{x^\a}{2r^2}\d_{\b\ol{\mu}}\th^{\ol{\mu}}\Bigr)+\th\wedge\frac{\sqrt{-1}}{4r^2}{\d^\a}_{\ol{\g}}\th^{\ol{\g}}+O(|x^\prime|^2).
\]
It follows from this equation that the TW torsion and the TW connection form satisfy 
\begin{equation}\label{torsion_at_p}
A_{\a\g}=-\frac{\sqrt{-1}}{4r^2}\d_{\a\g}+O(|x^\prime|^2)
\end{equation}
and 
\begin{equation*}
{\omega_\b}^\a=-\frac{\sqrt{-1}}{4r^2}\th{\d_\b}^\a+\frac{x^\a}{2r^2}\d_{\b\ol{\mu}}\th^{\ol{\mu}}+{B_{\g\b}}^\a\th^\g+O(|x^\prime|^2)
\end{equation*}
with some tensor ${B_{\g\b}}^\a={B_{\b\g}}^\a$. One can determine ${B_{\g\b}}^\a$ to be $(x^\a/2r^2)\d_{\g\b}$ from the condition $dh_{\a\ol{\b}}={\omega _{\a}}^{\g}h_{\g\ol{\b}}+h_{\a\ol{\g}}{\omega _{\ol{\b}}}^{\ol{\g}}$. The curvature form at $p$ is given by

\begin{multline*}
{\Omega_{\b}}^\a(p)=\frac{1}{4r^2}\Bigl(\d_{\g\ol{\mu}}{\d_{\b}}^\a+\d_{\b\ol{\mu}}{\d_{\g}}^\a-\d_{\b\g}{\d_{\ol{\mu}}}^\a\Bigr)\th^\g\wedge\th^{\ol{\mu}} \\
+\frac{1}{4r^2}{\d_{\g}}^\a\d_{\b\mu}\th^\g\wedge\th^{\mu}+\frac{1}{4r^2}{\d_{\ol{\g}}}^\a\d_{\b\ol{\mu}}\th^{\ol{\g}}\wedge\th^{\ol{\mu}}.
\end{multline*}
Therefore we have 
\begin{equation}\label{curvature_at_p}
{{R _{\b}}^{\a}}_{\g\ol{\mu}}=\frac{1}{4r^2}\Bigl(\d_{\g\ol{\mu}}{\d_{\b}}^\a+\d_{\b\ol{\mu}}{\d_{\g}}^\a-\d_{\b\g}{\d_{\ol{\mu}}}^\a\Bigr),\ \  
{\rm Ric}_{\g\ol{\mu}}=\frac{1}{2r^2}\d_{\g\ol{\mu}},\ \ 
\scal=\frac{1}{r^2}.
\end{equation}
The second equation implies that $\th$ is pseudo-Einstein. Using (\ref{five-dim-mu}), (\ref{torsion_at_p}) and (\ref{curvature_at_p}) together with the fact that $\th\wedge(d\th)^2=16r\cdot vol_{\pa\Omega^\prime_r}$, we obtain
\begin{equation*}
\mu(\pa\Omega_r)=\mu(\pa\Omega^\prime_r)=-\frac{20\pi}{27}\frac{1}{r^3}.
\end{equation*}
Now we recall the following theorem of Fefferman:
\begin{thm}[\cite{Fe1}]
Let $\Omega$ and $\Omega^\prime$ be strictly pseudoconvex domains. If $\Omega$ and $\Omega^\prime$ are biholomorphic, then 
$\pa\Omega$ and $\pa\Omega^\prime$ are CR diffeomorphic.
\end{thm}
Since $\mu(\pa\Omega_r)$ is a CR invariant, it follows that $\Omega_r$ and $\Omega_{r^\prime}$ are biholomorphic if and only if 
$r=r^\prime$.
\end{proof}

\end{document}